\documentclass{amsart}

\usepackage{xcolor}
\usepackage{setspace}

\newtheorem{theorem}{Theorem}[section]
\newtheorem{prop}[theorem]{Proposition}
\newtheorem{lem}[theorem]{Lemma}

\newtheorem{rem}[theorem]{Remark}

\numberwithin{equation}{section}
\begin{document}
\title[Fourier inequalities for $l^p$-balls]{Sharp Fourier inequalities and lattice point discrepancy for $l^p$-balls}

\author{Martin Lind}
\address{Department of Mathematics and Computer Science, Karlstad University, Universitetsgatan 2, 65188 Karlstad, Sweden}
\email{martin.lind@kau.se}



\begin{abstract}

For $1<p\le 2$, we establish sharp inequalities for the Fourier transform of the characteristic function of the $l^p$-unit ball $B_p\subset\mathbb{R}^2$. We show that
$$
\sup_{\boldsymbol{\omega}\in\mathbb{R}^2}\|\boldsymbol{\omega}\|_2^{3/2}|\widehat{\chi_{B_p}}(\boldsymbol{\omega})|\asymp (p-1)^{-1/2}\quad\text{as }p\rightarrow1+
$$
As an application, we obtain corresponding bounds for lattice point discrepancy inequalities for dilates of $B_p$.

\end{abstract}

\maketitle

\section{Introduction}

Inequalities for oscillatory integrals and Fourier transforms constitute a fundamental part of modern analysis, with applications ranging from harmonic analysis to geometric number theory. In many problems, the central difficulty is not only to establish decay rates, but also to control the associated constants in a precise and explicit manner. Such quantitative inequalities are particularly important when geometric parameters vary and approach critical or degenerate regimes.

A classical inequality of Hlawka \cite{Hlawka} asserts that if $K\subset\mathbb{R}^2$ is a bounded, convex domain with piecewise smooth boundary and everywhere non-vanishing curvature, then the Fourier transform of its characteristic function satisfies
\begin{equation}
    \label{decayFourier}
    |\widehat{\chi_K}(\boldsymbol{\omega})|\le \frac{C}{1+\|\boldsymbol{\omega}\|_2^{3/2}},\quad\boldsymbol{\omega}\in\mathbb{R}^2,
\end{equation}
for some constant $C>0$ depending on $K$, see also \cite{Herz,Randol1,Randol2}. (Here, $\|\boldsymbol{\omega}\|_2$ denotes the $l^2$-norm of $\boldsymbol{\omega}$.)
This estimate is optimal with respect to the decay exponent and underlies a variety of inequalities for lattice point discrepancies and related counting problems \cite{Iosevich1, Stein, SteinShak}. The curvature assumption is essential: if curvature vanishes, the inequality generally fails, indicating that the constant $C$ must encode fine geometric information.

In this work we study how such Fourier decay inequalities deteriorate as curvature degenerates, focusing on the unit ball $B_p$ of the $l^p$-norm in $\mathbb{R}^2$. For $1<p\le 2$ the the boundary of $B_p$ has strictly positive curvature, while at $p=1$ the curvature is a.e. 0 and (\ref{decayFourier}) fails. Hence, we expect the constant $C$ of (\ref{decayFourier}) to blow-up as $p\rightarrow1+$, and we are interested in determining its sharp asymptotic behavior as $p\rightarrow1+$.

Our main result establishes an explicit inequality of the form
\begin{equation}
    \label{limitingIneq}
    \sup_{\boldsymbol{\omega}\in\mathbb{R}^2}\|\boldsymbol{\omega}\|_2^{3/2}|\widehat{\chi_{B_p}}(\boldsymbol{\omega})|\le C(p-1)^{-1/2}, \quad 1<p\le 2,
\end{equation}
where $C$ is an absolute constant. We further show that (\ref{limitingIneq}) is sharp in the sense that the exponent $(p-1)^{-1/2}$ cannot be improved. The proof is entirely inequality-driven: it combines explicit oscillatory integral estimates with quantitative versions of van der Corput–type inequalities \cite{Rogers}, where all constants are tracked carefully.

We shall discuss an application of (\ref{limitingIneq}) to geometric number theory. Let $K\subset\mathbb{R}^2$ be a bounded, convex domain containing $\boldsymbol{0}$ and having a piecewise smooth boundary.
A classical problem is to study the asymptotic behavior of the number of lattice points in the scaled set $rK=\{r\boldsymbol{x}:\boldsymbol{x}\in K\}$ as $r$ grows large. Denote
\begin{equation}
    \nonumber
    N_K(r)=\sharp(\mathbb{Z}^2\cap rK),
\end{equation}
where $\sharp(A)$ denotes the cardinality of the finite set $A$. If $\partial K$ has non-vanishing curvature, then the lattice point discrepancy satisfies 
\begin{equation}
    \nonumber
    E_K(r)=|N_K(r)-\text{area}(K)r^2|=\mathcal{O}(r^{2/3}),
\end{equation}
see, e.g., \cite{Iosevich1, Stein, SteinShak}.
By the discussion above, we have for every $p\in(1,2]$ that
\begin{equation}
    \label{latticeIntr1}
    E_p(r)=E_{B_p}(r)\le C_pr^{2/3}.
\end{equation}
For $p=1$ the estimate (\ref{latticeIntr1}) fails (due to the vanishing of the curvature of $\partial B_1$). 
As an application of (\ref{limitingIneq}), we prove that 
\begin{equation}
    \label{latticeError}
    \sup_{r>0}\frac{E_p(r)}{r^{2/3}}\le C(p-1)^{-1/2}, \quad 1<p\le 2,
\end{equation}
where $C$ is an absolute constant. This provides another look on the failure of (\ref{latticeIntr1}) for $p=1$. In particular, the estimate (\ref{latticeError}) itself reflects the fact that $p=1$ is exceptional.

\subsection{Funding statement}
The author declares that no funds, grants, or other support were received during the preparation of this manuscript.

\subsection{Acknowledgments}
The author is deeply grateful to the three anonymous reviewers for their careful and thoughtful readings of my manuscript.
\section{Auxiliary results}
\subsection{Notations} 
For $p>0$ and $\boldsymbol{x}=(x,y)\in\mathbb{R}$ we denote 
$$
\|\boldsymbol{x}\|_p=\left(|x|^p+|y|^p\right)^{1/p}.
$$
Denote by 
$$
B_p(\boldsymbol{x}_0,r)=\{\boldsymbol{x}\in\mathbb{R}^2:\|\boldsymbol{x}-\boldsymbol{x}_0\|_p\le r\}.
$$
In particular, we denote by $B_p=B_p(0,1)$ the unit ball in $l^p$-norm. We shall only be interested in $B_p$ for $p\ge1$.

Let $\boldsymbol{x}=(x,y)$ and $\boldsymbol{\omega}=(\alpha,\beta)\in\mathbb{R}^2$. The Fourier transform of a function $f\in L^1(\mathbb{R}^2)$ is defined by
\begin{equation}
    \label{fourierDef}
    \widehat{f}(\boldsymbol{\omega})=\widehat{f}(\alpha,\beta)=\frac{1}{2\pi}\iint_{\mathbb{R}^2}e^{-i(x\alpha+y\beta)}f(x,y)dxdy.
\end{equation}
\subsection{Results on Fourier transform}

We first derive some simple formulae for the Fourier transform (\ref{fourierDef}) of functions of the form
\begin{equation}
    \label{almostRadial}
    f(\boldsymbol{x})=h(\|\boldsymbol{x}\|_p^p)\chi_{B_p}(\boldsymbol{x})
\end{equation}
where $p\ge1$ and $h:[0,1]\rightarrow[0,\infty)$.
\begin{lem}
Let $f$ given by (\ref{almostRadial}) and define $\varphi_p:[0,1]\rightarrow[0,1]$ by
\begin{equation}
\label{phi-Function}
\varphi_p(x)=\left(1-x^p\right)^{1/p}.
\end{equation}
then
\begin{equation}
\label{FT-almostRadial}
\widehat{f}(\alpha,\beta)=\frac{2}{\pi}\int_0^1\cos(\alpha x)\left(\int_0^{\varphi_p(x)}\cos(\beta y)h(x^p+y^p)dy\right)dx.
\end{equation}
\end{lem}
\begin{proof}
By Fubini's theorem, Euler's formula and even/odd functions, we have
\begin{eqnarray}
    \nonumber
    2\pi\widehat{f}(\alpha,\beta)&=&\iint_{B_p}e^{-i(x\alpha+y\beta)}h(|x|^p+|y|^p)dxdy\\
    \nonumber
    &=&\int_{-1}^1e^{-ix\alpha}\left(\int_{-(1-|x|^p)^{1/p}}^{(1-|x|^p)^{1/p}}e^{-iy\beta}h(|x|^p+|y|^p)dy\right)dx\\
    \nonumber
    &=&2\int_{-1}^1e^{-ix\alpha}\left(\int_0^{\varphi_p(|x|)}\cos(\beta y)h(|x|^p+y^p)dy\right)dx\\
    \nonumber
    &=&4\int_0^1\cos(\alpha x)\left(\int_0^{\varphi_p(x)}\cos(\beta y)h(x^p+y^p)dy\right)dx.
\end{eqnarray}
Whence (\ref{FT-almostRadial}) follows.
\end{proof}
\begin{rem}
\label{domainRemark}
From (\ref{FT-almostRadial}) and the fact that the cosine function is even,  it follows that
$$
\widehat{f}(\pm\alpha,\pm\beta)=\widehat{f}(\alpha,\beta)
$$
for any permutations of signs. Hence, we may assume that $\alpha,\beta\ge0$.
Furthermore, thanks to the symmetry in $x,y$ in (\ref{almostRadial}), we may without loss of generality assume that $\beta\ge\alpha$.
Consequently, it is sufficient to estimate $\widehat{f}(\alpha,\beta)$ on the sector
\begin{equation}
    \label{domainRest}
    R=\{(\alpha,\beta):0\le\alpha\le\beta\}.
\end{equation}
Note that for $\boldsymbol{\omega}=(\alpha,\beta)\in R$ there holds
\begin{equation}
    \label{domainIneq}
    \beta\le\|\boldsymbol{\omega}\|_2\le2\beta.
\end{equation}
\end{rem}
We shall use the following inequality that is a simple consequence of (\ref{FT-almostRadial}).
\begin{lem}
\label{weak-Estim-Lem}
Let $f$ be given by (\ref{almostRadial}) for any $p\ge1$ and assume that $h(1)=0$ and $h'\in L^1(0,1)$. Then
\begin{equation}
    \label{weak-Estim-FT}
    |\widehat{f}(\boldsymbol{\omega})|\le\frac{8\|h'\|_{L^1(0,1)}}{\pi(1+\|\boldsymbol{\omega}\|_2)}.
\end{equation}
\end{lem}
\begin{proof}
It sufficient to estimate $\widehat{f}(\boldsymbol{\omega})$ for $\boldsymbol{\omega}\in R$. First, note that
\begin{equation}
    \label{weak-Estim-FT-1}
    |\widehat{f}(\alpha,\beta)|\le\frac{2\|h\|_{L^\infty(0,1)}}{\pi}\le\frac{2\|h'\|_{L^1(0,1)}}{\pi}.
\end{equation}
We also have
\begin{equation}
    \nonumber
    |\widehat{f}(\alpha,\beta)|\le\frac{2}{\pi}\int_0^1dx\left|\int_0^{\varphi_p(x)}\cos(\beta y)h(x^p+y^p)dy\right|.
\end{equation}
Integration by parts of the inner integral gives
\begin{align}
\nonumber
\int_0^{\varphi_p(x)}\cos(\beta y)h(x^p+y^p)dy=\left[\frac{-\sin(\beta y)}{\beta}h(x^p+y^p)\right]_{y=0}^{y=\varphi_p(x)}\\
\nonumber
+\frac{1}{\beta}\int_0^{\varphi_p(x)}\sin(\beta y)h'(x^p+y^p)py^{p-1}dy\\
\nonumber
=\frac{1}{\beta}\int_0^{\varphi_p(x)}\sin(\beta y)h'(x^p+y^p)py^{p-1}dy,
\end{align}
since $h(1)=0$.
Hence, by (\ref{domainIneq})
\begin{eqnarray}
    \nonumber
    |\widehat{f}(\alpha,\beta)|&\le&\frac{2}{\pi\beta}\int_0^1dx\int_0^{\varphi_p(x)}|h'(x^p+y^p)|py^{p-1}dy\\
    \nonumber
    &=&\frac{2}{\pi\beta}\int_0^1dx\int_{x^p}^1|h'(u)|du\le \frac{2\|h'\|_{L^1(0,1)}}{\pi\beta}\\
    \label{weak-Estim-FT-2}
    &\le&\frac{4\|h'\|_{L^1(0,1)}}{\pi\|\boldsymbol{\omega}\|_2}.
\end{eqnarray}
Consequently, by (\ref{weak-Estim-FT-1}) and (\ref{weak-Estim-FT-2})
$$
|\widehat{f}(\boldsymbol{\omega})|\le\frac{4\|h'\|_{L^1(0,1)}}{\pi}\min\left(1,\frac{1}{\|\boldsymbol{\omega}\|_2}\right)\le\frac{8\|h'\|_{L^1(0,1)}}{\pi(1+\|\boldsymbol{\omega}\|_2)}.
$$
\end{proof}

\begin{lem}
\label{bumpLemma}
Define
\begin{equation}
\label{bump}
\Phi({\bf x})=c\exp\left(-\frac{1}{1-\|\boldsymbol{x}\|_2^2}\right)\chi_{B_2}(\boldsymbol{x}),
\end{equation}
where $c$ is taken so that $\int_{\mathbb{R}^2}\Phi(\boldsymbol{x})d\boldsymbol{x}=1$.
Then $\Phi\in C_0^\infty(\mathbb{R}^2)$ and
\begin{equation}
\label{bumpRate}
|\widehat{\Phi}(\boldsymbol{\omega})|\le\frac{3}{1+\|\boldsymbol{\omega}\|_2}.
\end{equation}
\end{lem}
\begin{rem}
In fact, $\widehat{\Phi}$ is rapidly decreasing since $\Phi\in C_0^\infty(\mathbb{R}^2)$. The point of (\ref{bumpRate}) is not the decay rate $\|\boldsymbol{\omega}\|_2^{-1}$ but rather the explicit value of the multiplicative constant.
\end{rem}
\begin{proof}[Proof of Lemma \ref{bumpLemma}]
We note that $\Phi(\boldsymbol{x})=ch(\|\boldsymbol{x}\|_2^2)\chi_{B_2}(\boldsymbol{x})$ where 
$$
h(t)=\exp(-1/(1-t))\chi_{[0,1]}(t).
$$
Since $h(1)=0$ we may apply Lemma \ref{weak-Estim-Lem} to get
\begin{equation}
\nonumber
|\widehat{\Phi}(\boldsymbol{\omega})|\le \frac{8c\|h'\|_{L^1(0,1)}}{\pi(1+\|\boldsymbol{\omega}\|_2)}\le\frac{8ce^{-1}}{\pi(1+\|\boldsymbol{\omega}\|_2)}.
\end{equation}
Further, it is easy to see that
\begin{equation}
    \nonumber
    c=\frac{1}{\pi\Gamma(-1,1)}\le3,
\end{equation}
where $\displaystyle\Gamma(u,v)=\int_v^\infty t^{u-1}e^{-t}\mathrm{d}t$ (the upper incomplete Gamma function). Thus,
\begin{equation}
    \nonumber
    |\widehat{\Phi}(\boldsymbol{\omega})|\le\frac{24e^{-1}}{\pi(1+\|\boldsymbol{\omega}\|_2)}\le\frac{3}{1+\|\boldsymbol{\omega}\|_2}.
\end{equation}   
\end{proof}
We shall use the following smooth approximation of the characteristic function of $B_p(0,r)$. Results of this type are well-known; in fact the lemma below is used implicitly in the proof of (\ref{latticeIntr1}) given in \cite[Theorem 8.2]{SteinShak}, but for our purposes it is necessary to know the value of constants.
\begin{lem}
\label{smoothCharLemma}
Let $p\in(1,2]$ and $r>1$ be fixed numbers and take arbitrary $\delta\in(0,1/2]$. Define
\begin{eqnarray}
    \label{smoothCharacteristic}
    \chi_{r,\delta}(\boldsymbol{x})=(\chi_{B_p(0,r)}*\Phi_\delta)(\boldsymbol{x}),
\end{eqnarray}
where $\Phi_\delta(\boldsymbol{x})=\delta^{-2}\Phi(\boldsymbol{x}/\delta)$ and $\Phi$ is given by (\ref{bump}). Set $c=2^{1/p-1/2}$. Then the following holds
\begin{enumerate}
    \item\label{1stItem} $\chi_{r,\delta}\in C^\infty(\mathbb{R}^2)$;
    \item\label{2ndItem} $\chi_{r,\delta}$ is supported on $B_p(0,r+c\delta)$;
    \item\label{3rdItem} $0\le \chi_{r,\delta}(\boldsymbol{x})\le 1$ for all $\boldsymbol{x}\in\mathbb{R}^2$;
    \item\label{4thItem} for any $\boldsymbol{x}\in\mathbb{R}^2$
    \begin{equation}
    \chi_{r-c\delta,\delta}(\boldsymbol{x})\le \chi_{B_p(0,r)}(\boldsymbol{x})\le \chi_{r+c\delta,\delta}(\boldsymbol{x}).
    \end{equation}
\end{enumerate}
\end{lem}
\begin{proof}
By definition and the fact that $\Phi$ is supported on $B_2(0,1)$ we have
\begin{equation}
    \label{ConvolDef}
    \chi_{r,\delta}(\boldsymbol{x})=\frac{1}{\delta^2}\int_{\|\boldsymbol{z}\|_2\le\delta}\Phi(\boldsymbol{z}/\delta)\chi_{B_p(0,r)}(\boldsymbol{x}-\boldsymbol{z})d\boldsymbol{z}.
\end{equation}
Property (\ref{1stItem}) is immediate since $\Phi$ is smooth.
We proceed with (\ref{2ndItem}).
Note that for $p\in[1,2]$ we have, with $c=2^{1/p-1/2}$, that
\begin{equation}
    \label{holderIneq}
    \|\boldsymbol{z}\|_2\le\|\boldsymbol{z}\|_p\le c\|\boldsymbol{z}\|_2.
\end{equation}
Fix arbitrary 
$\boldsymbol{x}\notin B_p(0,r+c\delta)$. Using the right inequality of (\ref{holderIneq}), we obtain  that for any  $\boldsymbol{z}\in B_2(0,\delta)$ there holds
$$
\|\boldsymbol{x}-\boldsymbol{z}\|_p\ge\|\boldsymbol{x}\|_p-\|\boldsymbol{z}\|_p\ge\|\boldsymbol{x}\|_p-c\|\boldsymbol{z}\|_2\ge r+c\delta-c\delta=r.
$$
In other words, $\boldsymbol{x}-\boldsymbol{z}\notin B_p(0,r)$ for any $\boldsymbol{z}\in B_2(0,\delta)$, whence (\ref{ConvolDef}) vanishes.
Property (\ref{3rdItem}) is clear: the integrand of (\ref{ConvolDef}) is non-negative and moreover
$$
\chi_{r,\delta}(\boldsymbol{x})\le\frac{1}{\delta^2}\int_{\|\boldsymbol{z}\|_2\le\delta}\Phi(\boldsymbol{z}/\delta)d\boldsymbol{z}=1.
$$
Finally, we prove (\ref{4thItem}). First, consider $\chi_{B_p(0,r)}\le \chi_{r+c\delta,\delta}$. It is sufficient to show that $\chi_{r+c\delta,\delta}(\boldsymbol{x})=1$ if $\|\boldsymbol{x}\|_p\le r$. 
For any $\boldsymbol{x}\in B_p(0,r)$ and $\boldsymbol{z}\in B_2(0,\delta)$, we have
$$
\chi_{B_p(0,r+c\delta)}(\boldsymbol{x}-\boldsymbol{z})=1.
$$
Indeed,
$$
\|\boldsymbol{x}-\boldsymbol{z}\|_p\le\|\boldsymbol{x}\|_p+\|\boldsymbol{z}\|_p\le\|\boldsymbol{x}\|_p+c\|\boldsymbol{z}\|_2\le r+c\delta.
$$
Thus, for any $\boldsymbol{x}\in B_p(0,r)$, (\ref{ConvolDef}) implies
$$
\chi_{r+c\delta,\delta}(\boldsymbol{x})=\int_{\|\boldsymbol{z}\|_2\le\delta}\Phi(\boldsymbol{z}/\delta)\chi_{B_p(0,r+c\delta)}(\boldsymbol{x}-\boldsymbol{z})d\boldsymbol{z}=\int_{\|\boldsymbol{z}\|_2\le\delta}\Phi(\boldsymbol{z}/\delta)d\boldsymbol{z}=1.
$$
Consider now $\chi_{r-c\delta,\delta}\le \chi_{B_p(0,r)}$. By (\ref{2ndItem}), $\chi_{r-c\delta,\delta}$ is supported on $B_p(0,r-c\delta+c\delta)= B_p(0,r)$ (note that $c\delta\le1<r)$.
Furthermore, by (\ref{3rdItem}), $\chi_{r-c\delta,\delta}\le1$ and it follows that for every $\boldsymbol{x}$ there holds $\chi_{r-c\delta,\delta}(\boldsymbol{x})\le\chi_{B_p(0,r)}(\boldsymbol{x})$.
\end{proof}

\subsection{On the function $\varphi_p(x)$}
In this subsection, we provide some additional information concerning the function (\ref{phi-Function}).
\begin{lem}
\label{FourierBallLemma}
For $(\alpha,\beta)\in R$ with $\beta>0$ we have
\begin{equation}
    \label{FourierBallEq1}
    \widehat{\chi_{B_p}}(\alpha,\beta)=\frac{2}{\pi\beta}\int_0^1\cos(\alpha x)\sin(\beta\varphi_p(x))dx.
\end{equation}
\end{lem}
\begin{proof}
Follows immediately from (\ref{FT-almostRadial}).
\end{proof}

\begin{lem}
\label{x-p-lem}
Let $\varphi_p$ be defined by (\ref{phi-Function}). The function $|\varphi_p''|$ has a unique minimum $x^*= x^*(p)$ on $[0,1]$. The minimum value satisfies
\begin{equation}
    \label{phi-p2ndDerivMin1}
    |\varphi_p''(x^*)|=\min_{0\le x\le1}|\varphi_p''(x)|=(p-1)m(p),
\end{equation}
where $m(p)$ is a decreasing function on $[1,2]$ with $m(1)=4$, $m(2)=1$. Further,
\begin{equation}
    \label{x-p-lem1}
    \frac{-1}{\varphi_p'(x^*)}\ge1,
\end{equation}
and 
\begin{equation}
    \label{x-p-lem2}
    \lim_{p\rightarrow1+}\frac{-1}{\varphi_p'(x^*)}=1.
\end{equation}
\end{lem}
\begin{proof}
We have
\begin{equation}
    \label{1stDervi-phi-p}
    \varphi_p'(x)=-x^{p-1}(1-x^p)^{1/p-1},
\end{equation}
\begin{equation}
    \label{2ndDervi-phi-p}
    \varphi_p''(x)=-(p-1)x^{p-2}(1-x^p)^{1/p-2},
\end{equation}
and
\begin{equation}
    \label{3rdDervi-phi-p}
    \varphi_p^{(3)}(x)=-(p-1)x^{p-3}(1-x^p)^{1/p-3}\left(x^p(p+1)+p-2\right).
\end{equation}
It is clear that the third derivative (\ref{3rdDervi-phi-p}) has only the zero
\begin{equation}
    \label{x-star}
    x^*\equiv x^*(p)=\left(\frac{2-p}{p+1}\right)^{1/p}
\end{equation}
on $(0,1)$. Further, $\varphi_p^{(3)}(x)>0$ for $x<x^*$ and $\varphi_p^{(3)}(x)<0$ for $x>x^*$. Hence, $\varphi_p''$ has a unique interior maximum at $x^*$. Since $\varphi_p''$ is strictly negative, it follows that $|\varphi_p''|$ has a unique interior minimum at $x^*$.
Inserting the expression (\ref{x-star}) into (\ref{2ndDervi-phi-p}) and taking absolute value yields
$$
|\varphi_p''(x^*)|=(p-1)(2-p)^{1-2/p}(2p-1)^{1/p-2}(p+1)^{1+1/p}.
$$
Define
$$
m(p)=(2-p)^{1-2/p}(2p-1)^{1/p-2}(p+1)^{1+1/p}\quad p\in[1,2), 
$$
and
$$
m(2)=\lim_{p\rightarrow2-}m(p).
$$
Clearly $m(1)=4$. By logarithmic differentiation,
$$
\frac{dm}{dp}=\frac{m(p)}{p^2}(-\ln(p+1)-\ln(2p-1)+2\ln(2-p))<0
$$
for $p\in(1,2)$. To demonstrate $m(2)=1$, we calculate the limit defining $m(2)$:
\begin{eqnarray}
\nonumber
\lim_{p\rightarrow2-}m(p)&=&\lim_{p\rightarrow2-}(2-p)^{1-2/p}\times 3^{-3/2}\times3^{3/2}\\
\nonumber
&=&\lim_{p\rightarrow2-}\exp\left(\left(1-\frac{2}{p}\right)\ln(2-p)\right)\\
\nonumber
&=&\lim_{t\rightarrow0+}\exp\left(\frac{-t\ln(t)}{(2-t)}\right)=\exp(0)=1.
\end{eqnarray}
To prove (\ref{x-p-lem1}), we insert the expression (\ref{x-star}) into (\ref{1stDervi-phi-p}) and get
$$
\varphi_p'(x^*)=-\left(\frac{2-p}{2p-1}\right)^{1-1/p}\quad\Leftrightarrow\quad\frac{-1}{\varphi_p'(x^*)}=\left(\frac{2p-1}{2-p}\right)^{1-1/p}.
$$
Since $p\ge1$, we have $(2p-1)/(2-p)\ge1$, thus proving (\ref{x-p-lem1}). Finally, (\ref{x-p-lem2}) follows by direct evaluation.
\end{proof}

\subsection{Oscillatory integrals}
We shall need van der Corput's lemma, which is a fundamental tool in the theory of oscillatory integrals, see e.g. \cite[Chapter VIII, §1.2]{Stein}. See also \cite{Rogers} and the references given therein for a discussion concerning sharp constants. We mention that Lemma \ref{vdCLemma} below is a special case of van der Corput's lemma, sufficient for our purpose of deriving (\ref{mainTeo:est1}).
\begin{lem}[van der Corput's lemma]
\label{vdCLemma}
Let $\psi:[a,b]\rightarrow\mathbb{R}$ be smooth on $(a,b)$ and let $r>0$.

If $\psi$ is monotone and $|\psi'(x)|\ge\lambda$ for all $x\in(a,b)$, then 
    \begin{equation}
        \label{vdC:est1}
        \left|\int_a^b\sin(r\psi(x))dx\right|\le \frac{2}{r\lambda}.
    \end{equation}
If $|\psi''(x)|\ge\lambda$ for all $x\in(a,b)$, then
     \begin{equation}
        \label{vdC:est2}
        \left|\int_a^b\sin(r\psi(x))dx\right|\le \frac{6}{\sqrt{r\lambda}}.
    \end{equation}

\end{lem}
\begin{rem}
\label{intervalRemark}
Below we shall use Lemma \ref{vdCLemma} for both the interval $[a,b]=[0,1]$ as well as certain sub-intervals of $[0,1]$; it is important that the constants at the right-hand sides of  (\ref{vdC:est1}) and (\ref{vdC:est2}) are independent of the interval $[a,b]$.
\end{rem}
The next lemma is closely related to the ``method of stationary phase''  \cite[Chapter VIII, §1.3]{Stein}. We shall use it as a kind of ``reverse van der Corput inequality'' to obtain (\ref{mainTeo:est2}). 
Rather than deriving Lemma \ref{statPhaseLemma} from general methods presented in e.g. \cite{Stein}, we give a self-contained proof.
\begin{lem}
\label{statPhaseLemma}
Let $\psi:[0,1]\rightarrow\mathbb{R}$ be a smooth function on $(0,1)$. Assume that there exists $x_0\in(0,1)$ such that $\psi(x_0)=\psi'(x_0)=0$,
\begin{equation}
    \label{statPhaseLemmaCond}
    |\psi''(x_0)|=\min_{0\le x\le1}|\psi''(x)|=\lambda>0,
\end{equation}
and $\psi^{(3)}$ is bounded in a neighbourhood of $x_0$. Then
\begin{equation}
    \label{reverseVDC}
    \left|\int_0^1\sin(r\psi(x))dx\right|=\frac{\sqrt{\pi}}{\sqrt{r\lambda}}+o(r^{-1/2})
\end{equation}
as $r\rightarrow\infty$.
\end{lem}
\begin{proof}
Without loss of generality, we may assume in the proof that $\psi''(x)>0$ for $x\in[0,1]$, so that the absolute values can be dropped from (\ref{statPhaseLemmaCond}). If $\psi''(x)<0$ for $x\in[0,1]$, then we simply replace $\psi$ with $-\psi$ and observe that the left-hand side of (\ref{reverseVDC}) remains unchanged, since the sine function is odd. 

Take small $\varepsilon>0$ (to be specified later). Write $[0,1]=\mathcal{J}_1\cup\mathcal{J}_2\cup\mathcal{J}_3$ where $\mathcal{J}_1=[0,x_0-\varepsilon]$, $\mathcal{J}_2=[x_0-\varepsilon,x_0+\varepsilon]$ and $\mathcal{J}_3=[x_0+\varepsilon,1]$. Denote further
$$
I_i=\int_{\mathcal{J}_i}\sin(r\psi(x))dx,\quad i=1,2,3.
$$
The main term is $I_2$; we estimate it first.
By Taylor's formula,
\begin{equation}
    \label{taylor}
    \psi(x)=\frac{\lambda (x-x_0)^2}{2}+\mathcal{O}((x-x_0)^3).
\end{equation}
Using (\ref{taylor}), a change of variable, and the mean value theorem, we get
\begin{eqnarray}
    \nonumber
    I_2&=&\int_{-\varepsilon}^\varepsilon\sin\left(\frac{r\lambda x^2}{2}+\mathcal{O}(rx^3)\right)dx\\
    \label{est1}
    &=&\int_{-\varepsilon}^\varepsilon\sin\left(\frac{r\lambda x^2}{2}\right)dx+\mathcal{O}(r\varepsilon^4).
\end{eqnarray}
Performing the change of variable $t=x\sqrt{r\lambda/2}$ in the integral at the right-hand side of (\ref{est1}) gives
\begin{equation}
    \label{I1estim}
I_2=\sqrt{\frac{2}{r\lambda}}\int_{-C\varepsilon\sqrt{r}}^{C\varepsilon\sqrt{r}}\sin(t^2)dt+\mathcal{O}(r\varepsilon^4),
\end{equation}
where $C=\sqrt{\lambda/2}$. We proceed to estimate $I_1$ and $I_3$. Since $\psi''(x)>0$ for all $x\in[0,1]$, the derivative $\psi'$ is increasing. Hence, $\psi'(x)\le\psi'(x_0-\varepsilon)<0$ for $x\in\mathcal{J}_1$. Furthermore,
$$
-\psi'(x_0-\varepsilon)=\int_{x_0-\varepsilon}^{x_0}\psi''(x)dx\ge\lambda\varepsilon
$$
whence
$$
\psi'(x)\le-\lambda\varepsilon,\quad x\in\mathcal{J}_1.
$$
Consequently, $\psi'$ satisfies the lower bound $|\psi'(x)|\ge\lambda\varepsilon$ for $x\in\mathcal{J}_1$. Since $\psi'$ also is increasing on $\mathcal{J}_1$, (\ref{vdC:est1}) and Remark \ref{intervalRemark} yield
\begin{equation}
    \label{I2estim}
    |I_1|\le \frac{2}{r\lambda\varepsilon}.
\end{equation}
A similar argument shows that (\ref{I2estim}) also holds for $I_3$.
Hence, by (\ref{I1estim}) and (\ref{I2estim})
\begin{equation}
    \nonumber
    \int_0^1\sin(r\psi(x))dx=\sqrt{\frac{2}{r\lambda}}\int_{-C\varepsilon\sqrt{r}}^{C\varepsilon\sqrt{r}}\sin(t^2)dt+\mathcal{O}(r\varepsilon^4)+I_1+I_3
\end{equation}
where $I_1,I_3$ satisfy the bound (\ref{I2estim}).
To deal with the fist term of the above equality, we use asymptotics for the Fresnel integral (see \cite[Chapter 7.3]{AS}):
\begin{equation}
    \label{fresnel}
    \int_{-m}^m\sin(x^2)dx=\sqrt{\frac{\pi}{2}}+\mathcal{O}\left(\frac{1}{m}\right),
\end{equation}
as $m\rightarrow\infty$. Hence, under the assumption that $\varepsilon\sqrt{r}$ is large, (\ref{fresnel}) yields
\begin{equation}
    \label{I1estim2}
    \sqrt{\frac{2}{r\lambda}}\int_{-C\varepsilon\sqrt{r}}^{C\varepsilon\sqrt{r}}\sin(t^2)dt=\frac{\sqrt{\pi}}{\sqrt{r\lambda}}+\mathcal{O}\left(\frac{1}{r\lambda \varepsilon}\right).
\end{equation}
Using (\ref{I1estim}), (\ref{I2estim}) and (\ref{I1estim2}), we get
$$
\int_0^1\sin(r\psi(x))dx=\frac{\sqrt{\pi}}{\sqrt{r\lambda}}+\mathcal{O}(r\varepsilon^4)+\mathcal{O}\left(\frac{1}{r\lambda \varepsilon}\right),
$$
again under the assumption that $\varepsilon\sqrt{r}$ is large. Take $\varepsilon=r^{-7/16}$, in this case $\varepsilon\sqrt{r}=r^{1/16}\rightarrow\infty$. Further, as $r\rightarrow\infty$,
$$
r\varepsilon^4=r^{1-7/4}=r^{-3/4}=o(r^{-1/2}),
$$
and
$$
\frac{1}{r\varepsilon}=\frac{1}{r^{1-7/16}}=\frac{1}{r^{9/16}}=o(r^{-1/2}).
$$
This concludes the proof of (\ref{reverseVDC}).
\end{proof}

\section{The sharp Fourier inequality}

We state the main theorem of this paper.
\begin{theorem}
\label{MainTeo:1p2}
Let $1<p\le2$. There is an absolute constan $C_1$ such that 
\begin{equation}
    \label{mainTeo:est1}
    \sup_{\boldsymbol{\omega}\in\mathbb{R}^2}\|\boldsymbol{\omega}\|_2^{3/2}|\widehat{\chi_{B_p}}(\boldsymbol{\omega})|\le\frac{C_1}{\sqrt{p-1}}.
\end{equation}
The estimate (\ref{mainTeo:est1}) is sharp in the following sense: there is an absolute constant $C_2$ such that for any $p\in(1,2)$, there exists a sequence $\{\boldsymbol{\omega}_n\}\subset\mathbb{R}^2$ with $\|\boldsymbol{\omega}_n\|_2\rightarrow\infty$ such that
\begin{equation}
    \label{mainTeo:est2}
    \|\boldsymbol{\omega}_n\|_2^{3/2}|\widehat{\chi_{B_p}}(\boldsymbol{\omega}_n)|\ge\frac{C_2}{\sqrt{p-1}}+o(1),
\end{equation}
where $o(1)$ means a term tending to $0$ as $n\rightarrow\infty$.
\end{theorem}
\begin{rem}
The constants $C_1,C_2$ can be taken to be
$$
C_1=12(2)^{1/4}\approx 14.270\ldots, \quad C_2=2^{7/4}\sqrt{\pi}\approx 5.961\ldots.
$$
\end{rem}
\begin{rem}
The estimate (\ref{mainTeo:est1}) has a "blow-up" of the order $(p-1)^{-1/2}$ as $p\rightarrow1+$, and (\ref{mainTeo:est2}) shows that this order is essentially sharp.
\end{rem}

We start with proving (\ref{mainTeo:est1}). It is sufficient to estimate $\widehat{\chi_{B_p}}$ on the sector $R$ given by (\ref{domainRest}), for this we shall use (\ref{FourierBallEq1}). Furthermore, proving (\ref{mainTeo:est1}) in the case $\alpha=0$ is easier, hence we assume $\beta\ge\alpha>0$.
It is useful to express the frequency variable $\boldsymbol{\omega}=(\alpha,\beta)$ in terms of polar coordinates, i.e.
$$
\alpha=r\cos(\theta),\quad \beta=r\sin(\theta).
$$
Then $(\alpha,\beta)\in R$ translates to
$$
r>0,\quad \theta\in[\pi/4,\pi/2).
$$
Abusing notation slightly, we write
$$
\widehat{\chi_{B_p}}(r,\theta)=\widehat{\chi_{B_p}}(r\cos(\theta),r\sin(\theta)),
$$
and by (\ref{FourierBallEq1})
$$
\widehat{\chi_{B_p}}(r,\theta)=\frac{2}{\pi r\sin(\theta)}\int_0^1\cos(r\cos(\theta)x)\sin(r\sin(\theta)\varphi_p(x))dx.
$$
Using the identity $2\cos(x)\sin(y)=\sin(y+x)+\sin(y-x)$, we get
\begin{equation}
    \label{fourierExpression}
    \widehat{\chi_{B_p}}(r,\theta)=\frac{1}{\pi r\sin(\theta)}\int_0^1 \left[\sin(r\psi_p(x,\theta))+\sin(r\widetilde{\psi}_p(x,\theta))\right]dx
\end{equation}
for $(r,\theta)\in(0,\infty)\times[\pi/4,\pi/2)$ where
$$
\psi_p(x;\theta)=\cos(\theta)x+\sin(\theta)\varphi_p(x),
$$
$$
\widetilde{\psi}_p(x;\theta)=-\cos(\theta)x+\sin(\theta)\varphi_p(x).
$$
\begin{rem}
The functions $\psi_p,\widetilde{\psi}_p$ are considered as functions of $x\in [0,1]$; $\theta$ is viewed as a parameter.
\end{rem}
\begin{proof}[Proof of (\ref{mainTeo:est1})] 
By (\ref{fourierExpression}),
\begin{eqnarray}
\nonumber
\|\boldsymbol{\omega}\|_2^{3/2}|\widehat{\chi_{B_p}}(\boldsymbol{\omega})|&=&\frac{\sqrt{r}}{\pi|\sin(\theta)|}\left|\int_0^1 \left[\sin(r\psi_p(x,\theta))+\sin(r\widetilde{\psi}_p(x,\theta))\right]dx\right|.
\end{eqnarray}
Using the above identity, the triangle inequality and the fact that $2^{-1/2}\le\sin(\theta)\le1$, (since $(\alpha,\beta)\in R$), we see that it suffices to show the existence of a constant $C_1>0$ such that
\begin{equation}
    \label{mainTeo:est1Help1}
    \sqrt{r}\left|\int_0^1\sin(r\psi_p(x;\theta))dx\right|+\sqrt{r}\left|\int_0^1\sin(r\widetilde{\psi}_p(x;\theta))dx\right|\le\frac{C_1}{\sqrt{p-1}},
\end{equation}
for any $p\in(1,2]$ and  $(r,\theta)\in(0,\infty)\times[\pi/4,\pi/2)$. We only estimate the first term at the left-hand side (\ref{mainTeo:est1Help1}); the argument is the same for the second term. Note that
$$
\min_{0\le x\le1}|\psi''_p(x;\theta)|=|\sin(\theta)|\min_{0\le x\le1}|\varphi''_p(x)|\ge\frac{m(p)}{\sqrt{2}}(p-1)\ge\frac{(p-1)}{\sqrt{2}},
$$
by (\ref{phi-p2ndDerivMin1}). Hence, by (\ref{vdC:est2})
$$
\sqrt{r}\left|\int_0^1\sin(r\psi_p(x;\theta))dx\right|\le\sqrt{r}\frac{6}{\sqrt{r(p-1)/\sqrt{2}}}=\frac{6(2^{1/4})}{\sqrt{p-1}}.
$$
Thus, we have shown that (\ref{mainTeo:est1Help1}) holds for any $p\in(1,2]$ and any $(r,\theta)\in(0,\infty)\times[\pi/4,\pi/2)$, with $C_1=12(2)^{1/4}$; this concludes the proof of (\ref{mainTeo:est1}).
\end{proof}

\begin{proof}[Proof of (\ref{mainTeo:est2})]

Fix $p\in(1,2)$, we shall describe how to construct the sequence $\{\boldsymbol{\omega}_n\}$ of (\ref{mainTeo:est2}). Let $x^*\in(0,1)$ be the point provided by Lemma \ref{x-p-lem}.
Define
$$
\theta^*=\arctan\left(\frac{-1}{\varphi'_p(x^*)}\right),
$$
by (\ref{x-p-lem1}) we have $\theta^*\in[\pi/4,\pi/2)$.
Note also that by taking $\theta=\theta^*$ in $\psi_p(x;\theta)$, we have $\psi'_p(x^*;\theta^*)=0$. Indeed,
$$
\psi'_p(x^*;\theta^*)=\cos(\theta^*)+\sin(\theta^*)\varphi'_p(x^*)=\cos(\theta^*)+\sin(\theta^*)\frac{-1}{\tan(\theta^*)}=0.
$$
Define now $\psi(x)=\psi_p(x;\theta^*)-\psi_p(x^*;\theta^*)$, by construction $\psi$ satisfies the conditions of Lemma \ref{statPhaseLemma} with $x_0=x^*$, and
$$
\min_{0\le x\le1}|\psi''(x)|=|\psi''(x^*)|=\sin(\theta^*)(p-1)m(p)\le4\sin(\theta^*)(p-1).
$$
Hence,
\begin{eqnarray}
\nonumber
\int_0^1\sin(r\psi(x))dx&=&\frac{\sqrt{\pi}}{\sqrt{r\sin(\theta^*)(p-1)m(p)}}+o(r^{-1/2})\\
\label{mainTeo:est2Help1}
&\ge&\frac{\sqrt{\pi}}{2\sqrt{r\sin(\theta^*)(p-1)}}+o(r^{-1/2}).
\end{eqnarray}

Further, 
$$
\widetilde{\psi}'_p(x;\theta^*)=-\cos(\theta^*)+\sin(\theta^*)\varphi_p'(x),
$$
and since $\varphi'_p(x)<0$, $\widetilde{\psi}_p(x;\theta^*)$ is a decreasing function of $x\in[0,1]$. Furthermore,
$$
\widetilde{\psi}'_p(x;\theta^*)\le-\cos(\theta^*).
$$
By (\ref{x-p-lem2}), $\theta^*\rightarrow\pi/4$ as $p\rightarrow1+$. Hence, we may assume that $p$ is sufficiently close to 1 in order to have $\theta^*\le\pi/3$. Therefore, $\widetilde{\psi}'_p(x;\theta^*)\le-\cos(\theta^*)\le-1/2$ so $|\widetilde{\psi}'_p(x;\theta^*)|\ge1/2$. By (\ref{vdC:est2})
\begin{equation}
    \label{mainTeo:est2Help2}
    \left|\int_0^1\sin(r\widetilde{\psi}_p(x;\theta^*))dx\right|\le\frac{4}{r}.
\end{equation}
Using (\ref{fourierExpression}), (\ref{mainTeo:est2Help2}) together with the definition of $\psi$, we have
\begin{eqnarray}
\nonumber
r\sin(\theta^*)\widehat{\chi_{B_p}}(r,\theta^*)&=&\int_0^1\sin(r\psi(x)+r\psi_p(x^*;\theta^*))dx
+\int_0^1\sin(r\widetilde{\psi}_p(x;\theta^*)dx\\
\nonumber
&=&\int_0^1\sin(r\psi(x)+r\psi_p(x^*,\theta^*))dx+\mathcal{O}\left(\frac{1}{r}\right).
\end{eqnarray}
Take now $r_n=2\pi n/\psi_p(x^*,\theta^*)$, then
$$
\sin(r_n\psi(x)+r_n\psi_p(x^*,\theta^*))=\sin(r_n\psi(x))
$$
and by (\ref{mainTeo:est2Help1})
\begin{eqnarray}
\nonumber
r_n\sin(\theta^*)\widehat{\chi_{B_p}}(r_n,\theta^*)&=&\int_0^1\sin(r_n\psi(x))dx+\mathcal{O}\left(\frac{1}{r_n}\right)\\
\nonumber
&\ge&\frac{2\sqrt{\pi}}{\sqrt{r_n\sin(\theta^*)(p-1)}}+o(r_n^{-1/2})+\mathcal{O}(r_n^{-1}).
\end{eqnarray}
Consequently, since $\sin(\theta^*)\ge1/\sqrt{2}$, we have
$$
r_n^{3/2}\widehat{\chi_{B_p}}(r_n,\theta^*)\ge\frac{2(2^{1/2})^{3/2}\sqrt{\pi}}{\sqrt{p-1}}+o(1),
$$
as $n\rightarrow\infty$.
Finally, setting $\boldsymbol{\omega}_n=r_n(\cos(\theta^*),\sin(\theta^*))$, the above relation states exactly that
$$
\|\boldsymbol{\omega}_n\|_2^{3/2}|\widehat{\chi_{B_p}}(\boldsymbol{\omega}_n)|\ge\frac{C_2}{\sqrt{p-1}}+o(1),
$$
which proves (\ref{mainTeo:est2}) with $C_2=2^{1+3/4}\sqrt{\pi}$.
\end{proof}

\section{Application to lattice point discrepancy}

In this section, we establish the following result on the lattice point discrepancy.
\begin{prop}\label{latticeProp}
Let $1<p\le2$. There is an absolute constant $C_3$ such that
   \begin{equation}
    \label{lattice}
    \sup_{r>0}\frac{E_p(r)}{r^{2/3}}\le\frac{C_3}{\sqrt{p-1}}.
\end{equation}
\end{prop}

Regarding the proof of Proposition \ref{latticeProp}, we claim no originality besides (\ref{mainTeo:est1}). In fact, we adapt the argument that is used to prove (\ref{latticeIntr1}) in e.g. \cite[Theorem 8.2]{SteinShak}. The only issue is that in order to get (\ref{lattice}), it is necessary to control the involved constants. Thus, every estimate must be explicit; to this end, we use Lemma \ref{weak-Estim-Lem}, Lemma \ref{bumpLemma} and Lemma \ref{smoothCharLemma}.

\begin{proof}[Proof of Proposition \ref{latticeProp}]
Let $\boldsymbol{n}=(m,n)\in\mathbb{Z}^2$, clearly
$|m|^p+|n|^p\le r^p$ if and only if $\chi_{B_p(0,r)}(\boldsymbol{n})=1$.
Consequently 
$$
N_p(r)=N_{B_p}(r)=\sum_{\boldsymbol{n}\in \mathbb{Z}^2}\chi_{B_p(0,r)}(\boldsymbol{n}).
$$
Let $p\in(1,2]$, let $r>1$ be fixed but arbitrary and take $\delta\in(0,1/2]$. (In fact, $r$ will be taken large and $\delta=\delta(r)$ will be taken small, see below.) Let also $c=2^{1/p-1/2}$ as in Lemma \ref{smoothCharLemma}. The "smooth counting function" is given by
\begin{equation}
    \nonumber
    N_{p,\delta}(r)=\sum_{\boldsymbol{n}\in\mathbb{Z}^2}\chi_{r,\delta}(\boldsymbol{n}).
\end{equation}
By Lemma \ref{smoothCharLemma}, property (\ref{4thItem}), there holds
\begin{equation}
    \label{lattice-Est2}
    N_{p,\delta}(r-c\delta)\le N_p(r)\le N_{p,\delta}(r+c\delta).
\end{equation}
Since $\chi_{r,\delta}\in C_0^\infty(\mathbb{R}^2)$ (by Lemma \ref{smoothCharLemma}, properties (\ref{1stItem}), (\ref{2ndItem})),  the Poisson summation formula applies
\begin{equation}
    \nonumber
    N_{p,\delta}(r)=\sum_{\boldsymbol{\omega}\in\mathbb{Z}^2}\widehat{\chi_{r,\delta}}(\boldsymbol{\omega})=\sum_{\boldsymbol{\omega}\in\mathbb{Z}^2}\widehat{\chi_{B_p(0,r)}}(\boldsymbol{\omega})\widehat{\Phi}(\delta\boldsymbol{\omega}).
\end{equation}
Note that $\chi_{B_p(0,r)}(\boldsymbol{x})=\chi_{B_p}(\boldsymbol{x}/r)$ whence
$$
\widehat{\chi_{B_p(0,r)}}(\boldsymbol{\omega})=\widehat{\chi_{B_p}(\cdot/r)}(\boldsymbol{\omega})=r^2\widehat{\chi_{B_p}}(r\boldsymbol{\omega}).
$$
Hence
\begin{eqnarray}
    \nonumber
    N_{p,\delta}(r)&=&
    r^2\widehat{\chi_{B_p}}(\boldsymbol{0})\widehat{\Phi}(\boldsymbol{0})+\\
    \nonumber&+&r^2\left(\sum_{0<|\boldsymbol{\omega}|\le\delta^{-1}}\widehat{\chi_{B_p}}(r\boldsymbol{\omega})\widehat{\Phi}(\delta\boldsymbol{\omega})+\sum_{|\boldsymbol{\omega}|>\delta^{-1}}\widehat{\chi_{B_p}}(r\boldsymbol{\omega})\widehat{\Phi}(\delta\boldsymbol{\omega})\right)\\
    \nonumber
    &=&\text{area}(B_p)r^2+r^2\left(\sum_{0<|\boldsymbol{\omega}|\le\delta^{-1}}\widehat{\chi_{B_p}}(r\boldsymbol{\omega})\widehat{\Phi}(\delta\boldsymbol{\omega})+\sum_{|\boldsymbol{\omega}|>\delta^{-1}}\widehat{\chi_{B_p}}(r\boldsymbol{\omega})\widehat{\Phi}(\delta\boldsymbol{\omega})\right).
\end{eqnarray}
Using (\ref{mainTeo:est1}) and (\ref{bumpRate}), we have
\begin{eqnarray}
    \nonumber    
    |N_{p,\delta}(r)-\text{area}(B_p) r^2|&\le&
    \frac{C_1r^2}{\sqrt{p-1}}\left(\sum_{0<\|\boldsymbol{\omega}\|_2\le\delta^{-1}}\frac{1}{r^{3/2}\|\boldsymbol{\omega}\|_2^{3/2}}+\sum_{\|\boldsymbol{\omega}\|_2>\delta^{-1}}\frac{3}{\delta r^{3/2}\|\boldsymbol{\omega}\|_2^{5/2}}\right)\\
    \nonumber
    &\le&\frac{C_1r^2}{\sqrt{p-1}}\left(\frac{1}{r^{3/2}}\int_{\|\boldsymbol{\omega}\|_2\le\delta^{-1}+1}\frac{d\boldsymbol{\omega}}{\|\boldsymbol{\omega}\|^{3/2}_2}+\right.\\
    \nonumber
    &+&\left.\frac{3}{\delta r^{3/2}}\int_{\|\boldsymbol{\omega}\|_2>\delta^{-1}-1}\frac{d\boldsymbol{\omega}}{\|\boldsymbol{\omega}\|^{5/2}_2}\right)\\
    \nonumber
    &=&\frac{C_1r^2}{\sqrt{p-1}}\left(\frac{4\pi(\delta^{-1}+1)^{1/2}}{r^{3/2}}+\frac{12\pi}{\delta r^{3/2}(\delta^{-1}-1)^{1/2}}\right).
\end{eqnarray}
For the integrals in the penultimate line above, we used
$$
\int_{\|\boldsymbol{\omega}\|_2\le R}\frac{d\boldsymbol{\omega}}{\|\boldsymbol{\omega}\|^{3/2}_2}=4\pi R^{1/2}\quad\text{and}\quad\int_{\|\boldsymbol{\omega}\|_2>R}\frac{d\boldsymbol{\omega}}{\|\boldsymbol{\omega}\|^{5/2}_2}=4\pi R^{-1/2}.
$$
Since $\delta\le1/2$, we have $(\delta^{-1}+1)^{1/2}\le\sqrt{2}\delta^{-1/2}$ and $(\delta^{-1}-1)^{-1/2}=\delta^{1/2}(1-\delta)^{1/2}\le 2^{3/2}\delta^{-3/2}$. Whence,
\begin{eqnarray}
    \label{lattice-Est1}
    |N_{p,\delta}(r)-\text{area}(B_p) r^2|&\le&4\pi C_1\left(\sqrt{2}+3\sqrt{8}\right)\frac{r^{1/2}\delta^{-1/2}}{\sqrt{p-1}}.
\end{eqnarray}
In fact, (\ref{lattice-Est1}) may be written
\begin{equation}
    \label{lattice-Est11}
    N_{p,\delta}(r)=\text{area}(B_p) r^2+\frac{r^{1/2}\delta^{-1/2}g(r)}{\sqrt{p-1}},
\end{equation}
where $g(r)$ is a function that remains bounded for all $r>1$. 
Using $r\pm c\delta$ in (\ref{lattice-Est11}) yields the estimates
\begin{eqnarray}
    \label{lattice-Est21}
    N_{p,\delta}(r+c\delta)=\text{area}(B_p) r^2+\frac{r^{1/2}\delta^{-1/2}g(r+c\delta)}{\sqrt{p-1}}-\epsilon^+(r,\delta),
\end{eqnarray}
and
\begin{equation}
    \label{lattice-Est22}
    N_{p,\delta}(r-c\delta)=\text{area}(B_p) r^2+\frac{r^{1/2}\delta^{-1/2}g(r-c\delta)}{\sqrt{p-1}}-\epsilon^-(r,\delta),
\end{equation}
where $\max\{|\epsilon^+(r,\delta)|,|\epsilon^-(r,\delta)|\}\le 4\text{area}(B_p)cr\delta$.
By (\ref{lattice-Est2}), (\ref{lattice-Est21}) and (\ref{lattice-Est22})
we have
\begin{equation}
    \label{lattice-Est3}
     E_p(r)\le\frac{C r^{1/2}\delta^{-1/2}}{\sqrt{p-1}}+4\text{area}(B_p)cr\delta\le\frac{C}{\sqrt{p-1}}(r^{1/2}\delta^{-1/2}+r\delta),
\end{equation}
since $(p-1)^{-1/2}\ge1$.
Taking $\delta=r^{-1/3}$ in (\ref{lattice-Est3}) completes the proof of (\ref{lattice}), with $C_3=8\pi C_1\times 7\sqrt{2}=21\pi 2^{23/4}\approx 3550$.
\end{proof}

\end{document}